\documentclass[11pt]{amsart}

\title{S-stable foliations on flow-spines with transverse Reeb flow}

\author[Handa]{Shin Handa}
\address{Mathematical Institute, Tohoku University, Sendai, 980-8578, Japan}
\email{shin.handa.t5@dc.tohoku.ac.jp}

\author[Ishikawa]{Masaharu Ishikawa}
\address{Department of Mathematics, Hiyoshi Campus, Keio University, 
4-1-1 Hiyoshi, Kohoku, Yokohama 223-8521, Japan}
\email{ishikawa@keio.jp}

\usepackage{amsfonts,amsmath,amssymb,amscd}
\usepackage{amsthm}
\usepackage{latexsym}
\usepackage{graphicx}
\usepackage{psfrag}
\usepackage{color}

\theoremstyle{plain}
\newtheorem*{theorem*}{Theorem}
\newtheorem*{lemma*} {Lemma}
\newtheorem*{corollary*} {Corollary}
\newtheorem*{proposition*}{Proposition}
\newtheorem*{conjecture*}{Conjecture}
\newtheorem{theorem}{Theorem}[section]
\newtheorem{lemma}[theorem]{Lemma}
\newtheorem{corollary}[theorem]{Corollary}
\newtheorem{proposition}[theorem]{Proposition}

\newtheorem{remark}[theorem]{Remark}

\theoremstyle{remark}

\newtheorem*{definition}{Definition}

\theoremstyle{definition}

\newtheoremstyle{citing}
  {}
  {}
  {\itshape}
  {}
  {\bfseries}
  {.}
  {.5em}
  {\thmnote{#3}}

\theoremstyle{citing}

\newcommand{\Real}{\mathbb{R}}

\newcommand{\Int}{\mathrm{Int}\,}

\newcommand{\Nbd}{\mathrm{Nbd}}

\newcommand{\pr}{\mathrm{pr}}

\makeatletter
\@addtoreset{equation}{section}

\makeatother

\begin{document}

\maketitle

\begin{abstract}
The notion of S-stability of foliations on branched simple polyhedrons is introduced by R. Benedetti and C. Petronio in the study of characteristic foliations of contact structures on $3$-manifolds.
We additionally assume that the $1$-form $\beta$ defining a foliation on a branched simple polyhedron $P$ 
satisfies $d\beta>0$, 
which means that the foliation is a characteristic foliation of a contact form whose Reeb flow is transverse to $P$. In this paper, we show that if there exists a $1$-form $\beta$ on $P$
with $d\beta>0$ then we can find a $1$-form with the same property and additionally being S-stable.
We then prove that the number of simple tangency points of an S-stable foliation on a positive or negative flow-spine is at least $2$ and give a recipe for constructing a characteristic foliation of a $1$-form $\beta$ with $d\beta>0$ on the abalone.
\end{abstract}

\section{Introduction}

A flow-spine $P$ is a branched simple spine embedded in an oriented, closed, smooth 3-manifold $M$ such that there exists a non-singular flow in $M$ that is transverse to $P$ and ``constant" in the complement $M\setminus P$. This notion was introduced by I.~Ishii in~\cite{Ish86}. In that paper, he proved that any non-singular flow in $M$ has a flow-spine. Therefore, regarding a Reeb flow on a contact 3-manifold as a flow of its flow-spine, we may use it for studying contact structures on 3-manifolds. 
This setting is analogous to the setting of the correspondence between contact $3$-manifolds and open book decompositions in~\cite{TW75, Gir02}. One of the advantages of this setting is that the contact structure in the complement of a flow-spine is always tight since the Reeb flow is  ``constant''. Thus the study of contact structures via flow-spines divides into two parts, one is the study of contact structures in small neighborhoods of flow-spines and the other is to see what happens by gluing a tight 3-ball to the neighborhoods.

A characteristic foliation on a branched polyhedron embedded in a contact 3-manifold had been studied by Benedetti and Petronio in~\cite{BP00}, 
following the work of Giroux on characteristic foliations on surfaces~\cite{Gir91}. 
Setting the branched polyhedron in a general position, we may assume that the foliation is non-singular on the singular set $S(P)$ of $P$. We further assume that the indices of simple tangency points of the foliation to $S(P)$ are always $+1$ and away from vertices of $S(P)$.
See Figure~\ref{fig3} for the definition of the index of a simple tangency point. A foliation that satisfies the above conditions is called an {\it S-stable} foliation. 
In~\cite{BP00}, they proved several statements. For instance, it is proved that if a characteristic foliation on a branched polyhedron $P$ in a contact $3$-manifold $M$ is S-stable then the contact structure with this foliation is unique in a small neighborhood of $P$ up to contactomorphism. It is also proved that if that contact structure is tight in a neighborhood of $P$ then it extends to a tight contact structure on $M$ and the extended contact structure on $M$ is unique up to contactomorphism. Remark that the Reeb flows of these contact structures may not be transverse to $P$. 
In this paper, we always assume that a contact structure is positive, that is, its contact form $\alpha$ satisfies $\alpha\land d\alpha>0$.

\begin{figure}[htbp]
\includegraphics[width=90mm, bb=163 641 431 712]{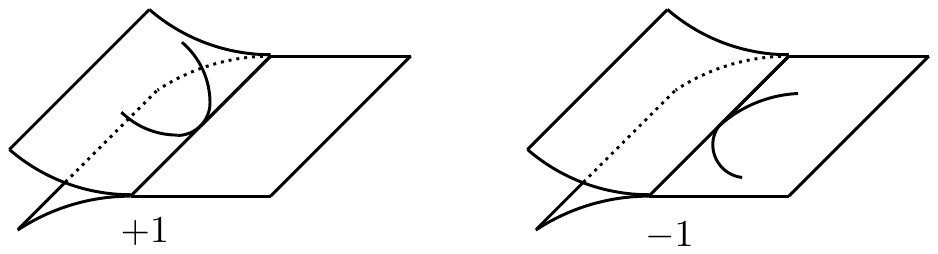}
\caption{The index of simple tangency points.}\label{fig3}
\end{figure}




Suppose that there exists a contact form $\alpha$ on $M$ 
whose Reeb flow is transverse to a flow-spine $P$.
We choose the orientations of the regions of $P$ so that their intersections with the Reeb flow are positive. 
In this setting, the form $\beta=\alpha|_P$ on $P$ satisfies $d\beta>0$. 
To advance the study in this setting, we need to study a $1$-form $\beta$ on $P$ with $d\beta>0$ and whose characteristic foliation on $P$ is S-stable.
Note that there are many branched simple polyhedrons that admit a $1$-form $\beta$
with $d\beta>0$, see Remark~\ref{rem_S}. 

The aim of this paper is to understand if there is a $1$-form $\beta$ on $P$ with $d\beta>0$ and whose kernel gives an S-stable foliation on $P$ and if there exists a constraint for positions of leaves of S-stable foliations on $P$. The following theorem answers the first question. 

\begin{theorem}\label{thm1}
Let $P$ be a branched simple polyhedron.
If there exists a $1$-form $\beta$ on $P$ with $d\beta>0$ then there exists a $1$-form $\beta'$ on $P$ such that $d\beta'>0$ and the foliation defined by $\beta'=0$ on $P$ is S-stable.
\end{theorem}


Note that it is proved in~\cite{BP00} that any characteristic foliation on a branched simple polyhedron $P$ in a contact $3$-manifold $M$ can be made to be S-stable by $C^0$-perturbation of $P$ in $M$. In our claim, there is no direct relation between $\beta$ and $\beta'$.

Our construction of an S-stable foliation is very efficient in the sense that the number of simple tangency points is very small 
(at most twice the number of triple lines). 
On the other hand, it is difficult to find an S-stable foliation defined by a $1$-form $\beta$ with $d\beta>0$ and without simple tangency points.
Actually, in Theorem~\ref{thm2} below, we show that such a foliation does not exist if a branched simple polyhedron is a flow-spine and satisfies a certain condition.
A region of a branched simple polyhedron is called a {\it preferred region} if the orientations of all edges and circles on its boundary are opposite to the one induced from the orientation of the region defined by the branching of $P$.
Note that the number of simple tangency points is always even, see Lemma~\ref{lemma_even}.

\begin{theorem}\label{thm2}
If a flow-spine $P$ has a preferred region then any foliation on $P$ defined by a $1$-form $\beta$ with $d\beta>0$ has at least two simple tangency points.
\end{theorem}

A point on a simple polyhedron that has a neighborhood shown in Figure~\ref{fig1}~(iii) is called a {\it vertex}.
A branched simple polyhedron has two kinds of vertices: 
the vertex shown on the middle in Figure~\ref{fig2} is called a {\it vertex of $\ell$-type}
and the one on the right is called {\it of $r$-type}.
A flow-spine $P$ is said to be {\it positive} if it has at least one vertex and all vertices are of $\ell$-type. 
In~\cite{IIKN20} it is shown that the map sending a positive flow-spine $P$ to 
the contact structure 
whose Reeb flow is a flow of $P$ gives a surjection from the set of positive flow-spines to the set of contact $3$-manifolds up to contactomorphism. It is also proved that we cannot expect the same result without restricting the source to the set of positive flow-spines. 
Thus, the positivity is important when we study contact $3$-manifolds using flow-spines.
We say that a flow-spine 
is {\it negative} if it has at least one vertex and all vertices are of $r$-type.
 
If a flow-spine is either positive or negative then it always has a preferred region. Hence the following corollary holds.

\begin{corollary}\label{cor1}
Let $P$ be a positive or negative flow-spine of a closed, oriented, smooth $3$-manifold $M$. 
If the Reeb flow of a contact form $\alpha$ on $M$ is a flow of $P$ 
then the characteristic foliation of $\ker\alpha$ on $P$ has at least two simple tangency points.
\end{corollary}

Remark that if $P$ is a positive flow-spine then there exists a contact form $\alpha$ on $M$   whose Reeb flow is a flow of $P$
and such a contact structure is unique up to contactomorphism, which is proved 
in~\cite[Theorem~1.1]{IIKN20}.
On the other hand, if $P$ is a negative flow-spine, we do not know if there exists such a contact form or not. 

As we mentioned, our second aim is to understand if there is a constraint for S-stable characteristic foliations in our setting, and the above corollary gives some insight into this question.
Furthermore, 
in Section~\ref{ex1}, we will give an example of an S-stable foliation given by a $1$-form $\beta$ with $d\beta>0$ on the abalone explicitly,
in which we can see that there exists a constraint for the positions of leaves of S-stable characteristic foliations, see Remark~\ref{rem62}.
We will also give a branched standard spine that admits an S-stable foliation defined by $\beta=0$ with $d\beta>0$ and without simple tangency points, see Section~\ref{ex3}. 
We do not know if there exists a flow-spine that admits an S-stable foliation defined by a $1$-form $\beta$ with $d\beta>0$ and without simple tangency points.

This paper is organized as follows. In Section~\ref{sec2}, we recall some terminologies of polyhedrons that we use in this paper. Section~\ref{sec3} and Section~\ref{sec4} are devoted to the proofs of Theorems~\ref{thm1} and~\ref{thm2}, respectively. In Section~\ref{sec5}, we shortly introduce the DS-diagram of a flow-spine and give the proof of Corollary~\ref{cor1}. 
In Section~\ref{sec6}, after giving a Poincar\'{e}-Hopf lemma for flow-spines, we give an example of an S-stable foliation given by a $1$-form $\beta$ with $d\beta>0$ on the abalone explicitly. We also give an example of a branched standard spine that admits an S-stable foliation given by a $1$-form $\beta$ with $d\beta>0$ and without simple tangency points.

We would like to thank Ippei Ishii for precious comments and especially for telling us the $3$-manifold of the spine in Figure~\ref{fig9}.
We are also grateful to Yuya Koda and Hironobu Naoe for useful conversation.
Finally, we thank the anonymous referee for insightful comments on improving the paper.  
This work is partially supported by JSPS KAKENHI Grant Number JP17H06128.
The second author is supported by JSPS KAKENHI Grant Numbers JP19K03499 
and Keio University Academic Development Funds for Individual Research.

\section{Preliminaries}\label{sec2}

In this section, we recall terminologies of polyhedrons used in this paper.

\subsection{Simple polyhedron}\label{sec21}

A polyhedron $P$ is said to be {\it simple} if every point on $P$ has a neighborhood represented by one of the models shown in Figure~\ref{fig1}.
Each connected component of the set of points with the model~(i) is called a {\it region},
that with the model~(ii) is called a {\it triple line} and that with the model~(iii) is called a {\it true vertex},
which we call a {\it vertex} for short. 
Let $S(P)$ denote the union of triple lines and vertices of $P$, which is called the {\it singular set} of $P$. 
A triple line is either an open arc, called an {\it edge}, or a circle.
A simple polyhedron $P$ is said to be {\it standard}\footnote{It is also called a {\it special} polyhedron.} if every connected component of $P\setminus S(P)$ is homeomorphic to an open disk and every triple line is an open arc.

\begin{figure}[htbp]
\includegraphics[width=120mm, bb=129 612 551 712]{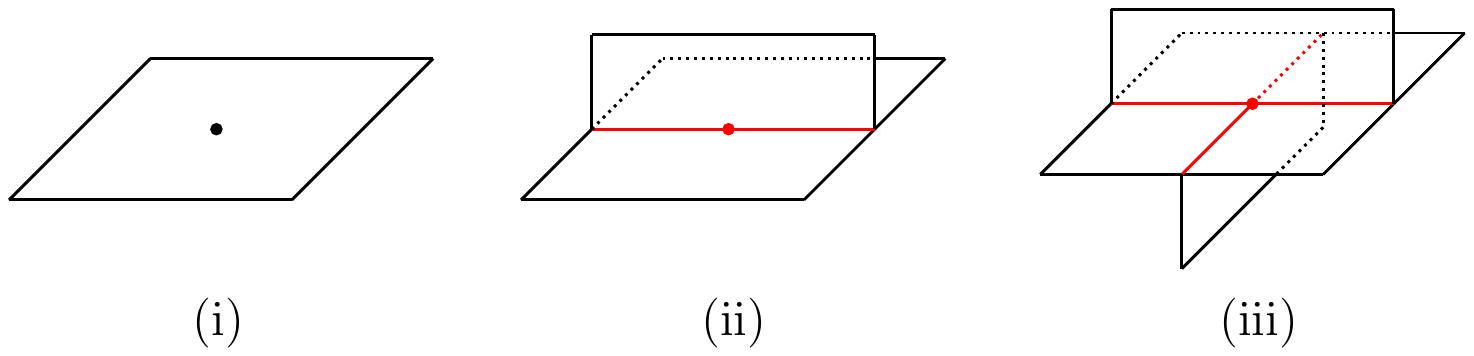}
\caption{The local models of a simple polyhedron.}\label{fig1}
\end{figure}

Let $P$ be a simple polyhedron embedded in an oriented, $3$-manifold $M$ and assume that each region of $P$ is orientable.
An assignment of orientations to the regions of $P$ such that for any triple line the three orientations induced from those of the adjacent regions do not coincide is called a {\it branching}.
A simple polyhedron equipped with a branching is called a {\it branched polyhedron}.
If a branched polyhedron is standard then it is called a {\it branched standard polyhedron}.
For a branched polyhedron $P$, we define the orientation of each triple line of $S(P)$ by the orientation induced from the two adjacent regions, see Figure~\ref{fig2}.

\begin{figure}[htbp]
\includegraphics[width=120mm, bb=129 619 584 712]{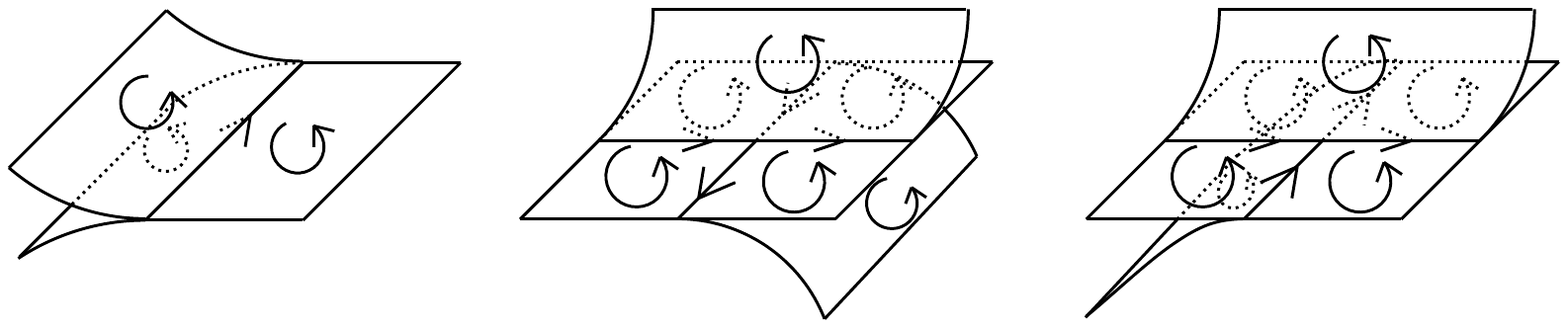}
\caption{The orientation of a triple line induced from the branching.}\label{fig2}
\end{figure}

A region of a branched simple polyhedron is called a {\it preferred region} if the orientations of all edges on its boundary are opposite to the one induced from the orientation of the region defined by the branching of $P$.

\subsection{Spine and flow-spine}

Let $M$ be a closed, connected, oriented $3$-manifold and $P$ be a simple polyhedron embedded in $M$. The polyhedron $P$ is called a {\it spine} of $M$ if  $M\setminus\text{Int\,} B$ collapses to $P$, where $B$ is a $3$-ball in $M\setminus P$. If a spine is standard (resp. branched) then it is called a {\it standard} (resp. {\it branched}) {\it spine}. 

The singular set $S(P)$ of a branched simple polyhedron $P$ 
can be regarded as the image of an immersion of a finite number of circles. 
The immersion has only normal 
crossings as shown in Figure~\ref{fig2}. 
A flow-spine is defined from a non-singular flow in a closed, connected, 
oriented, smooth $3$-manifold and a disk $D$ intersecting all orbits of the flow transversely
by floating the boundary of $D$ smoothly until it arrives in the disk $D$ itself.
See~\cite{Ish86} for the precise definition (cf.~\cite{IIKN20}). 
By the construction, the flow is positively transverse to the flow-spine. 
We can easily see that a branched simple polyhedron $P$ is a flow-spine if and only if 
$S(P)$ is the image of an immersion of one circle.

\subsection{Admissibility condition}

Let $P$ be a branched simple polyhedron. 
To have a $1$-form $\beta$ on $P$ with $d\beta>0$, 
$P$ need satisfy the following admissibility condition. 
Let $R_1,\ldots,R_n$ be the regions of $P$.
Regions and triple lines of $P$ are oriented by the branching.
Let $\bar R_i$ be the metric completion of $R_i$ with the path metric inherited from a Riemannian metric on $R_i$. 
Let $\kappa_i:\bar R_i\to P$ be the natural extension of the inclusion $R_i\to P$. 

\begin{definition}
A branched simple polyhedron $P$ is said to be {\it admissible} if there exists an assignment of real numbers $x_1,\ldots,x_m$ to the triple lines $e_1,\ldots,e_m$, respectively, of $P$ such that
for any $i\in\{1,\ldots,n\}$
\begin{equation}\label{eq_adm}
   \sum_{\tilde e_j\subset \partial \bar R_i}\varepsilon_{ij} x_j>0,
\end{equation}
where $\tilde e_j$ is an open arc or a circle on $\partial \bar R_i$ such that
$\kappa_i|_{\tilde e_j}:\tilde e_j\to e_j$ is a homeomorphism, and $\varepsilon_{ij}=1$ if the orientation of $e_j$ coincides with that of $\kappa_i(\tilde e_j)$ induced from 
the orientation of $R_i$ and $\varepsilon_{ij}=-1$ otherwise.
\end{definition}

\section{Proof of Theorem~\ref{thm1}}\label{sec3}

Theorem~\ref{thm1} will follow from the following proposition.

\begin{proposition}\label{thm_adm}
Let $P$ be a branched simple polyhedron.
Then, there exists a $1$-form $\beta$ on $P$ with $d\beta>0$ such that the foliation on $P$ defined by $\beta=0$ is S-stable if and only if $P$ is admissible.
\end{proposition}

\begin{remark}\label{rem_S}
It is known that many flow-spines satisfy the admissibility condition, see~\cite[Section~4.1]{IIKN20}.
For instance, a branched standard spine in a rational homology $3$-sphere is admissible. 
Such branched simple polyhedrons always satisfy the condition in Theorem~\ref{thm1}.
\end{remark}

Let $P$ be a branched simple polyhedron, $Q$ be a small compact neighborhood of $S(P)$ in $P$, 
$R_1',\ldots,R_n'$ be connected components of  $P\setminus\text{Int\,}Q$ and $e_1,\ldots,e_m$ be triple lines of $P$. The orientation of $R_i'$ is defined from the branching of $P$ and that of $e_j$ is defined as explained in Section~\ref{sec21}.

\begin{lemma}\label{lemma31}
Suppose that $P$ is admissible. Then there exists a $1$-form $\beta_0$ on $Q$ such that 
\begin{itemize}
\item[(1)] the foliation on $Q$ defined by $\beta_0=0$ is S-stable, 
\item[(2)] $\int_{\partial R_i'}\beta_0>0$ for $i=1,\ldots,n$, and
\item[(3)] $d\beta_0>0$ on $Q$.
\end{itemize}
\end{lemma}

\begin{proof}
Let $v_1,\ldots,v_{n_v}$ be the vertices of $P$ and $N_{v_j}$ be a small compact neighborhood of $v_j$ in $P$. For each $j=1,\ldots,n_v$, we define a projection $\pr_{v_j}$ from $N_{v_j}$ to $\Real^2$ such that 
\begin{itemize}
\item[(i)] the orientation of each region defined by the branching coincides with that of $\Real^2$,
\item[(ii)] the image is included in a small open disk centered at the point $\left(\cos \frac{2\pi j}{n_v}, \sin\frac{2\pi j}{n_v}\right)$, and
\item[(iii)] the orientations of the images of the triple lines in $N_{v_j}$ are counter-clockwise.
\end{itemize}
See Figure~\ref{fig4}. 
The arrowed edges in the figure are the images of triple lines of $P$, which are oriented 
according to the branching of $P$ as shown in Figure~\ref{fig2}. 

\begin{figure}[htbp]
\begin{center}
\includegraphics[width=8.0cm, bb=137 373 497 712]{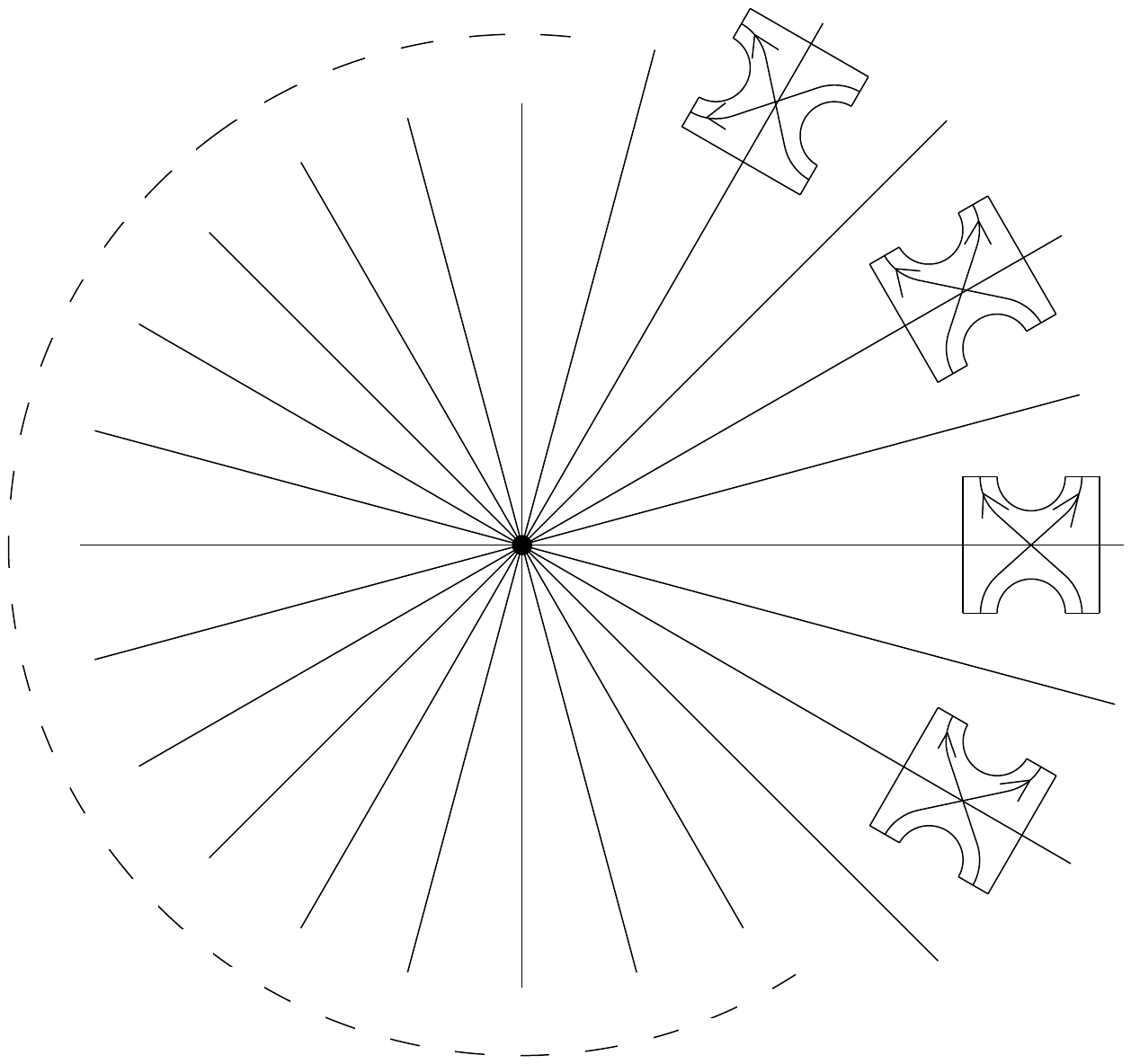}
\caption{The images of $N_{v_1},\ldots,N_{v_{n_v}}$ on $\Real^2$.}\label{fig4}
\end{center}
\end{figure}

Let $(r,\theta)$ be the polar coordinates of $\Real^2$, 
set a $1$-form on $\Real^2$ as $r^2d\theta$ and define the $1$-form $\beta_0$ on $N_{v_j}$ by
$\pr_{v_j}^*(r^2d\theta)$. Note that $d\beta_0>0$ on $N_{v_i}$ since $d(r^2d\theta)=2rdr\land d\theta>0$.

Next, we define the $1$-form $\beta_0$ on a neighborhood of each triple line of $P$.
Let $e_1,\ldots, e_m$ be the triple lines of $P$ and 
$N_{e_j}$, $j=1,\ldots,m$, be a small compact neighborhood of $e_j$ in $P$
and set $N_{e_j}'=N_{e_j}\setminus\Int(N_{v(e_j)}\cup N_{v'(e_j)})$,
where $v(e_j)$ and $v'(e_j)$ are the two vertices at the endpoints of $e_j$.
Since $P$ is assumed to be admissible, we can choose an $m$-tuple of real numbers $(E_1,\ldots,E_m)$ in $C(P)$. 

Suppose that $e_j$ is an edge.
If $E_j>0$, we choose a projection $\pr_{e_j}$ from $N'_{e_j}$ to $\Real^2$ as shown on the left in Figure~\ref{fig5} and define the $1$-form $\beta_0$ on $N'_{e_j}$ by
$\pr_{e_j}^*(r^2d\theta)$ as before. 
Set $e_j'=e_j\cap N'_{e_j}$.
Let $a_j$ and $a'_j$ be the endpoints of $\pr_{e_j}(e_j')$ and let $\ell_{e_j}$ and $\ell_{e_j}'$ be the line segments on $\Real^2$ connecting the origin and $a_j$ and the origin and $a_j'$, respectively.
Since these segments are in the kernel of $r^2d\theta$, 
the integrations of $r^2\theta$ along $\ell_{e_j}$ and $\ell_{e_j}'$ are $0$.
By Stokes' theorem, the absolute value of the integration of $\beta_0$ along $e_j'$ coincides with the area of the region bounded by $\pr_{e_j}(e'_j)$, $\ell_{e_j}$ and $\ell_{e_j}'$.
We may choose $\pr_{e_j}$ so that the area becomes the given real number $E_j>0$, which means that the integration of $\beta_0$ along $e'_j$ coincides with $E_j$.
Note that the foliation defined by $\beta_0=0$ has no simple tangency point on $N_{e_j}'$.

\begin{figure}[htbp]
\includegraphics[width=125mm, bb=129 594 470 711]{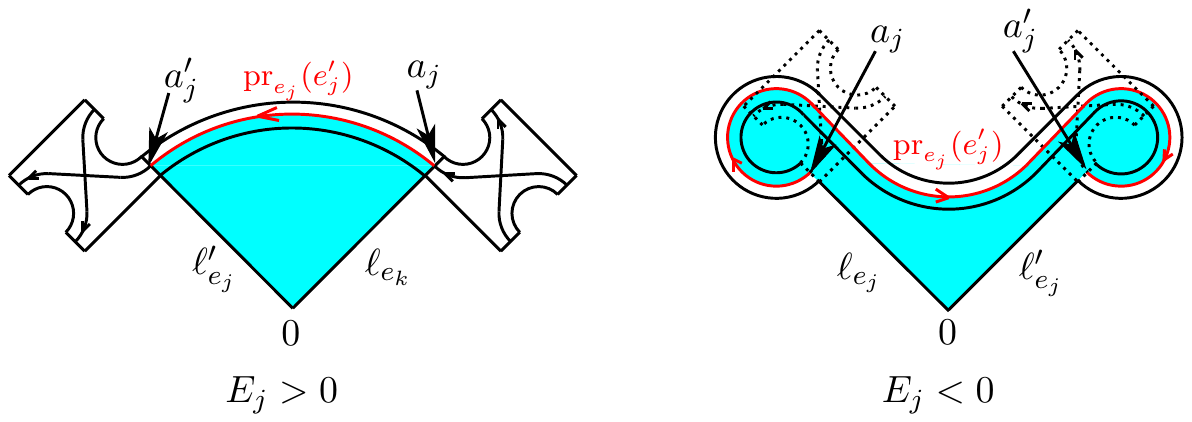}
\caption{The images of $N'_{e_j}$ on $\Real^2$ in the case $E_j>0$ and $E_j<0$.}\label{fig5}
\end{figure}

If $E_j<0$, we choose a projection $\pr_{e_j}$ from $N'_{e_j}$ to $\Real$ as shown on the right in Figure~\ref{fig5} and define the $1$-form $\beta_0$ on $N'_{e_j}$ by $\pr_{e_j}^*(r^2d\theta)$ as before. By the same observation as in the case $E_j>0$, we may choose $\pr_{e_j}$ so that the integration of $\beta_0$ along $e'_j$ coincides with $E_i$.
For each edge $e_j$ with $E_j<0$, the foliation defined by $\beta_0=0$ has exactly 
two simple tangency points and their indices are $+1$.

If $E_j=0$, we choose the projection $\pr_{e_j}$ as shown in Figure~\ref{fig6}. 
In the figure, $\pr_{e_j}$ is chosen in such a way that the area of the region $D_{e_j}^+$ bounded by $\pr_{e_j}(e'_j)$, $\ell_{e_j}$ and $\ell_{e_j}'$ is equal to the area of the region $D_{e_j}^-$ bounded by $\pr_{e_j}(e_j')$.
Then we have
\[
   E_j=\int_{\pr_{e_j}(e'_j)}r^2d\theta=\int_{D_{e_j}^+}2rdr\land d\theta
-\int_{D_{e_j}^-}2rdr\land d\theta=0.
\]
Two simple tangency points of index $+1$ appear in this case.

\begin{figure}[htbp]
\begin{center}
\includegraphics[width=6cm, bb=171 522 387 711]{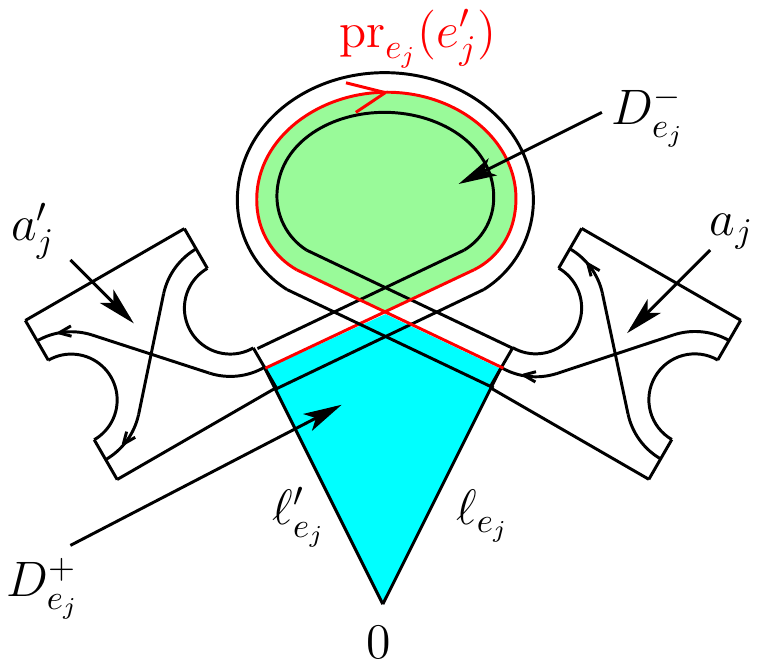}
\caption{The image of $N'_{e_j}$ on $\Real^2$ in the case $E_j=0$.}\label{fig6}
\end{center}
\end{figure}

If $e_j$ is a circle and $E_j\ne 0$ then we embed $N_{e_j}$ into
$\Real^2$ along a circle centered at the origin and define $\beta_0$ by 
$\pr_{e_j}^*(r^2d\theta)$ as before.
We choose the embedding such that the orientation of the triple line in the image is counter-clockwise if $E_j>0$ and clockwise if $E_j<0$.
Choosing the radius of the circle suitably, we have $\int_{e_j}\pr_{e_j}^*(r^2d\theta)=E_j$. 
It has no simple tangency point.

If $e_j$ is a circle and $E_j=0$ then we embed $N_{e_j}$ into $\Real^2$ along an ``$8$''-shaped immersed curve so that the origin is in the middle of the region whose boundary is oriented counter-clockwise. 
Using the same trick as in Figure~\ref{fig6}, we may have an embedding of $N_{e_j}$ such that
$\int_{e_j}\pr_{e_j}^*(r^2d\theta)=0$. Note that the foliation defined by $r^2d\theta=0$ has two simple tangency points of index $+1$.

Finally, we check if the $1$-form $\beta_0$ obtained above satisfies the required conditions.
It satisfies the condition~(1) by the construction.
The condition~(2) is satisfied by choosing $Q$ to be sufficiently narrow so that
$\int_{\partial R_i'}\beta_0$ is sufficiently close to 
$\sum_{j=1}^m \varepsilon_{ij}E_j$, where $\varepsilon_{ij}$ is the coefficient in inequality~\eqref{eq_adm}.
Since $(E_1,\ldots,E_m)\in C(P)$, the sum $\sum_{j=1}^m \varepsilon_{ij}E_j$ is positive. Hence
$\int_{\partial R_i'}\beta_0$ is also positive.
The condition~(3) is obviously satisfied since $d(r^2d\theta)>0$.
This completes the proof.
\end{proof}

Now we extend the $1$-form $\beta_0$ on $Q$ to $P$ by applying the following lemma used in the paper of Thurston and Winkelnkemper~\cite{TW75}.
For a compact surface $\Sigma$ with boundary $\partial \Sigma$, let $\Nbd(\partial \Sigma;\Sigma)$ denote
a small compact neighborhood of $\partial \Sigma$ in $\Sigma$.

\begin{lemma}[Thurston-Winkelnkemper]\label{lemma_TW}
Let $\Sigma$ be a compact, oriented surface with boundary and $\beta_0$ be a $1$-form on 
$\Nbd(\partial\Sigma;\Sigma)$ such that $\int_{\partial\Sigma}\beta_0>0$ and $d\beta_0>0$. Then there exists a $1$-form $\beta$ on $\Sigma$ such that
\begin{itemize}
\item $\beta=\beta_0$ on $\Nbd(\partial\Sigma;\Sigma)$, and
\item $d\beta>0$ on $\Sigma$.
\end{itemize}
\end{lemma}

\begin{proof}[Proof of Proposition~\ref{thm_adm}]
Suppose that a branched simple polyhedron $P$ is admissible. 
The conditions~(2) and~(3) for $\beta_0$ in Lemma~\ref{lemma31} allow us to use Lemma~\ref{lemma_TW} for each region $R_i$ of $P$. Then the $1$-form $\beta$ obtained by this lemma satisfies the required conditions in the assertion.

Conversely, if there exists a $1$-form $\beta$ that satisfies the required conditions then 
the $m$-tuple $\left(\int_{e_1}\beta,\ldots,\int_{e_m}\beta\right)$ is an element in $C(P)$ by Stokes' theorem.
Therefore $P$ is admissible. This completes the proof.
\end{proof}

\begin{proof}[Proof of Theorem~\ref{thm1}]
As mentioned in the end of the above proof, 
if there exists a $1$-form $\beta$ on $P$ with $d\beta>0$ then $P$ is admissible.
Thus Theorem~\ref{thm1} follows from Proposition~\ref{thm_adm}.
\end{proof}

\begin{remark}
The number of simple tangency points of the foliation defined by $\beta$ constructed in the above proof is at most $2m$, where $m$ is the number of triple lines of $P$.
\end{remark}

\section{Proof of Theorem~\ref{thm2}}\label{sec4}

In this section, we prove Theorem~\ref{thm2}

\begin{lemma}\label{lemma33}
Let $P$ be a branched simple polyhedron. If $P$ has a preferred region then
\[
   \{(x_1,\ldots,x_m)\in C(P)\mid x_i>0,\;\,i=1,\ldots,m\}
\]
is empty.
\end{lemma}

\begin{proof}
Let $\bar R$ be the metric completion of a preferred region $R$
and $\kappa:\bar R\to P$ be the natural extension of the inclusion $R\to P$.
Let $\{l_1,\cdots,l_s\}$ be the set of edges and circles on  the boundary $\partial \bar R$ of $\bar R$ such that $\kappa(l_k)$ is a triple line of $P$. 

Assume that there exists a point $(E_1,\ldots,E_m)$ in $C(P)$ with $E_i>0$ for $i=1,\ldots,m$.
For $k=1,\ldots,s$, let $L_k$ be the real number $E_{j_k}$ assigned to the triple line $e_{j_k}$ of $P$ with $\kappa(l_k)=e_{j_k}$. 
Since the orientation of $l_k$ is opposite to the orientation induced from that of $R$, the sum of 
the assigned real numbers along $\partial \bar R$ becomes $-L_1-\cdots -L_s<0$.
 This contradicts the assumption that $(E_1,\ldots,E_m)\in C(P)$.
\end{proof}

Let $\mathcal F$ be an S-stable foliation on a branched simple polyhedron $P$ defined by a $1$-form $\beta$ on $P$.
Since $\mathcal F$ is S-stable, it is transverse to $S(P)$ in a small neighborhood of a vertex of $P$.

\begin{lemma}\label{lemma_even}
Let $P$ be a branched simple polyhedron and $\mathcal F$ be an S-stable foliation on $P$ defined by a $1$-form $\beta$.
Then the number of simple tangency points of $\mathcal F$ is even.
\end{lemma}

\begin{proof}
The assertion follows from the fact that $S(P)$ consists of the image of an immersion of a finite number of circles and $\mathcal F$ is co-oriented.
\end{proof}

Since $\mathcal F$ is transverse to $S(P)$ in small neighborhoods of vertices, there are four kinds of projections as shown in Figure~\ref{fig7}. In the figure, type $1$ is the case where both of the oriented edges of $S(P)$ intersect the leaves of $\mathcal F$ in the same direction and type~$2$ is the case where
they intersect the leaves in opposite directions. The arrows with symbol $\beta$ represent the direction along which the integration of $\beta$ becomes positive.
The sign of type is $+$ if the direction of the arrow coincides with the orientation of the edge passing from the left-bottom to the right-top and the sign is $-$ otherwise. Each sign $+$ or $-$ written near the endpoints of the 
edges of $S(P)$ represents the sign of the value obtained by inserting a vector tangent to the edge
whose direction is consistent with the orientation of the edge into the $1$-form $\beta$.
In each figure, only the painted region can be a part of a preferred region 
in both of $\ell$-type and $r$-type vertex cases. 
We call the small neighborhoods of vertices in Figure~\ref{fig7} 
with labels $1_+$, $1_-$, $2_+$ and $2_-$
the {\it H-pieces of type $1_+$, $1_-$, $2_+$ and $2_-$}, respectively.

\begin{figure}[htbp]
\begin{center}
\includegraphics[width=10.0cm, bb=129 610 491 709]{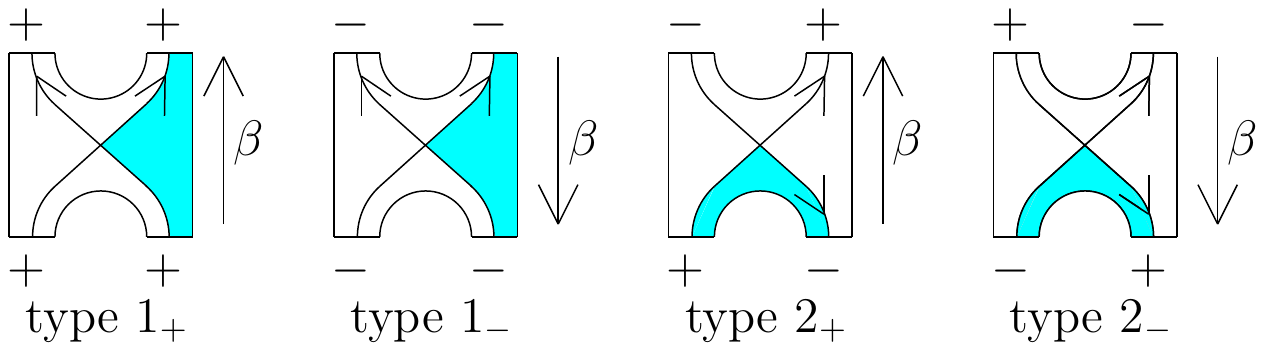}
\caption{H-pieces.}\label{fig7}
\end{center}
\end{figure}

\begin{lemma}\label{lemma41}
Let $P$ be a flow-spine of a closed, oriented, smooth $3$-manifold
and $\mathcal F$ be an S-stable foliation on $P$ defined by a $1$-form.
If there exists a vertex of $P$ whose H-piece is of type~$2$ then $\mathcal F$ has a simple tangency point.
\end{lemma}

Recall that $S(P)$ consists of the image of an immersion of a finite number of circles.
In the following proofs, a {\it circuit} means the image of an immersion of one of these circles.
If a circuit passes all 
triple lines then it is called an {\it Euler circuit}.
A branched simple spine is a flow-spine if and only if it has an Euler circuit.

\begin{proof}
Assume that there exists a vertex of $P$ whose H-piece is of type $2$ and $\mathcal F$ has no tangency point. The latter condition implies that if two edges of $S(P)$ are connected in a circuit then the integrations of $\beta$ along these 
edges are either both positive or both negative. Since there is an H-piece of type $2$, $S(P)$ has at least two circuits. This contradicts the assumption that $P$ is a flow-spine.
\end{proof}

\begin{proof}[Proof of Theorem~\ref{thm2}]
Assume that there exists an S-stable foliation $\mathcal F$ defined by a $1$-form $\beta$ with $d\beta>0$ and without simple tangency points. 
If $P$ has no vertex then it does not admit a $1$-form $\beta$ with $d\beta>0$. Suppose $P$ has at least one vertex. 
By Lemma~\ref{lemma41}, the H-pieces of all the vertices are of type $1$.
Furthermore, all vertices must have the same sign, otherwise there exists an edge of $P$
which connects two vertices with opposite signs. 
If all vertices are of type $1_+$ then $\int_{e_j}\beta>0$ for all edges $e_1,\ldots,e_m$.
Since $P$ has a preferred region, this is impossible by Lemma~\ref{lemma33}.
If all vertices are of type $1_-$ then $\int_{e_j}\beta<0$ for all edges $e_1,\ldots,e_m$.
Now we take the sum of the integrations of $\beta$ along the boundaries of all the regions of $P$.
In this calculation, $\int_{e_j}\beta$ appears twice and $\int_{-e_j}\beta$ once for each 
$j=1,\ldots,m$.  Since we have $\int_{e_j}\beta<0$ for all $j=1,\ldots,m$,
the sum of the integrations is negative. This contradicts the assumption that $d\beta>0$ on $P$.
\end{proof}

\section{DS-diagram and Proof of Corollary~\ref{cor1}}\label{sec5}

To explain the proof of Corollary~\ref{cor1}, we first shortly introduce the DS-diagram of a flow-spine,
see~\cite{Ish92, IIKN20} for details (cf.~\cite{II85}). Let $P$ be a flow-spine of a closed, oriented, smooth $3$-manifold $M$. The complement $M\setminus P$ is an open ball.
The singular set $S(P)$ of $P$ induces a trivalent graph on the boundary of the closed ball $B^3$ obtained as the geometric completion of $M\setminus P$, which is called the {\it DS-diagram} of $P$. 
The flow of $P$ induces a flow on $B^3$, and 
the set of points on the boundary $\partial B^3$ at which the flow is tangent to $\partial B^3$ constitutes a simple closed curve. 
This curve is called the {\it E-cycle}. 
The E-cycle separates $\partial B^3$ into two open disks $S^+$ and $S^-$, where the flow is positively transverse to $S^+$ and negatively transverse to $S^-$.
Here the orientation of $\partial B^3$ is induced from that of $B^3$, and
those of $S^+$ and $S^-$ are induced by the inclusion $S^+, S^-\subset \partial B$.
The E-cycle is oriented as the boundary of $S^+$. 
Observing the geometric completion of $B^3$ in $M$, we may see that each region of $P$ corresponds to a region on $S^+$ bounded by the DS-diagram and each triple line of $P$ corresponds to an edge or circle in $S^+$. The same is true for $S^-$. 
Each triple line of $P$ also corresponds to an edge or circle in the E-cycle. 
The orientation of $S^+$ (resp. $S^-$) is consistent with (resp. opposite to)  the orientations of the regions of $P$.
We assign orientations to the edges and circles of the DS-diagram so that they coincide
with the orientations of the triple lines of $P$. 

We call a region on $S^+$ not adjacent to the E-cycle an {\it internal region}. 
If a DS-diagram has a circle component then the DS-diagram consists of three parallel circles one of which is the E-cycle, one is in $S^+$ and the last one is in $S^-$. The $3$-manifold of the flow-spine given by this DS-diagram is $S^1\times S^2$. See~\cite[Remark 1.2]{EI05} (cf.~\cite[Example 3]{IIKN20}).

\begin{lemma}\label{lemma_disk}
The DS-diagram of a flow-spine has an internal region on $S^+$ homeomorphic to a disk.
\end{lemma}

\begin{proof}
Let $n_v$ be the number of vertices of $P$.
Observing the geometric completion of $M\setminus P$, we can verify that
the DS-diagram has $n_v$ vertices on $S^+$ and $2n_v$ vertices on the E-cycle. 
If the DS-diagram has a circle component then the assertion follows. 
We assume that it has no circle component. 
Let $G$ be the subgraph of the DS-diagram consisting of edges lying in $S^+$ and vertices adjacent to these edges. Note that the edges on the E-cycle are not contained in $G$. 
We denote by $\bar G$ the union of $G$ and the regions on $S^+$ bounded by $G$. 
Let $e(\bar G)$ and $f(\bar G)$ be the numbers of  edges and regions of $\bar G$, respectively.
Since $G$ has $n_v$ univalent vertices and $n_v$ trivalent vertices,
we have $n_v+3n_v=2e(\bar G)$. On the other hand, 
the Euler characteristic $\chi(\bar G)$ of $\bar G$ is given as 
\[
   \chi(\bar G)=2n_v-e(\bar G)+f(\bar G).
\]
These equalities imply that $f(\bar G)=\chi(\bar G)$.
Since $\chi(\bar G)>0$, we have $f(\bar G)>0$.
Hence there exists an internal region homeomorphic to a disk.
\end{proof}


%
%

\begin{proof}[Proof of Corollary~\ref{cor1}]
By Lemma~\ref{lemma_disk}, the DS-diagram of a flow-spine has an internal region homeomorphic to a disk.
It is proved in~\cite[Lemma~4.6]{IIKN20} that 
if a flow-spine is positive then an internal region is always a preferred region,
and the proof written there works for negative flow-spines also. 
Thus, the assertion follows from Theorem~\ref{thm2}.
\end{proof}

\section{Examples}\label{sec6}

In this section, we give an S-stable characteristic foliation given by a $1$-form $\beta$ with $d\beta>0$ on the abalone explicitly. We also give an example of branched standard spine that has an S-stable foliation without simple tangency points.

\subsection{Poincar\'{e}-Hopf}

We introduce a lemma that is useful for determining singularities of foliations on flow-spines. 
Let $P$ be a flow-spine of a closed, oriented, smooth $3$-manifold $M$
and $\mathcal F$ be a foliation on $P$ with only elliptic and hyperbolic singularities.
Assume that $\mathcal F$ is tangent to $S(P)$ at a finite number of points in the interior of triple lines and the singularities do not lie on $S(P)$. 
Let $e$ be the number of elliptic singularities, $h$ be the number of hyperbolic singularities and $t_+$ and $t_-$ be the 
numbers of simple tangency points of $\mathcal F$ with index $+1$ and $-1$, respectively.

\begin{lemma}\label{lemma_PH}
Suppose that the foliation $\mathcal F$ is in the above setting.
Then the following equality holds:
\[
   e-h=1+\frac{t_+-t_-}{2}.
\]
Note that if $\mathcal F$ is S-stable then $t_-=0$.
\end{lemma}

\begin{proof}
According to the definition of DS-diagram,
the regions and triple lines of $P$ can be described on the boundary of the closed $3$-ball $B^3$
obtained as the geometric completion of the complement $M\setminus P$. 
Thus, we may describe $\mathcal F$ on the boundary $\partial B^3$ of $B^3$,
which has $2e$ elliptic singularities and $2h$ hyperbolic singularities.
Each simple tangency point of index $+1$ becomes as shown on the top-middle in Figure~\ref{fig8}
and it can be regarded as a hyperbolic singularity as shown on the right. 
Each simple tangency point of index $-1$ becomes as shown on the bottom in the figure
and it can be regarded as an elliptic singularity.
By Poincar\'{e}-Hopf formula on the sphere $\partial B^3$,
we have $(2e+t_-)-(2h+t_+)=2$. Thus the assertion follows.
\end{proof}

\begin{figure}[htbp]
\begin{center}
\includegraphics[width=12.0cm, bb=129 544 471 712]{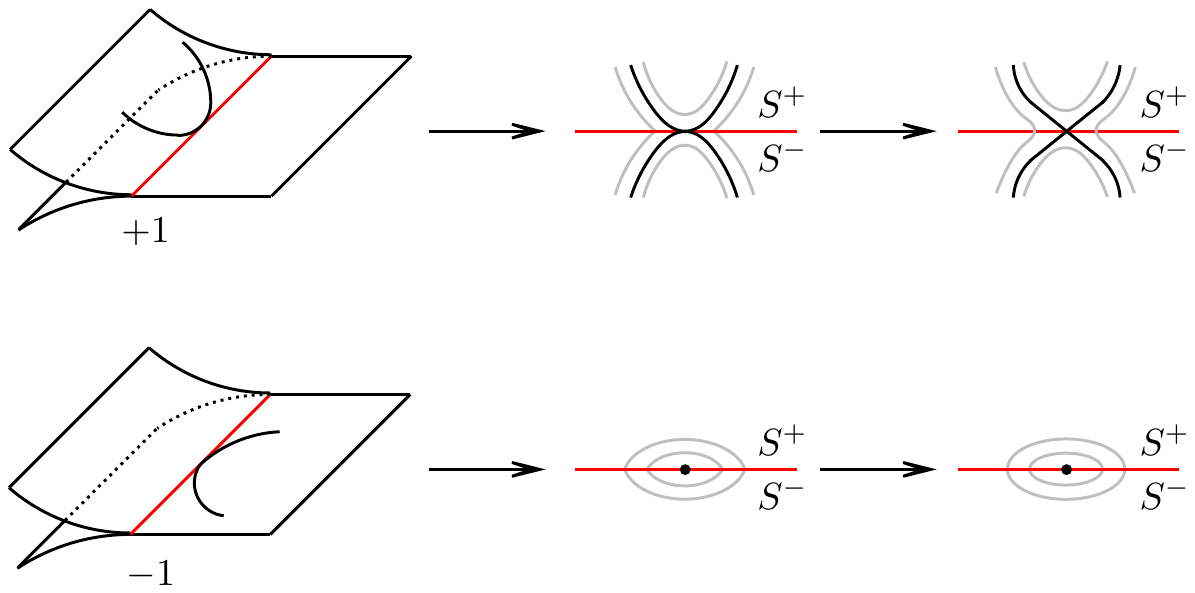}
\caption{Simple tangency points and corresponding singularities on $\partial B^3$.}\label{fig8}
\end{center}
\end{figure}

\subsection{Foliation on the abalone}\label{ex1}

The branched polyhedron $A$ obtained from the neighborhood of the singular set shown in Figure~\ref{fig11} by attaching two disks $D_1$ and $D_2$ is called the {\it abalone}. This is a flow-spine of $S^3$.
We can easily see that $(E_1,E_2)\in C(A)$ if and only if 
\[
\left\{
\begin{split}
& -E_1>0 \\
& 2E_1+E_2>0.
\end{split}
\right.
\]
This system of inequalities has a solution, for example 
$(-1,13)\in\Real^2$. 
Hence $A$ is admissible. 

\begin{figure}[htbp]
\begin{center}
\includegraphics[width=10.0cm, bb=129 528 551 712]{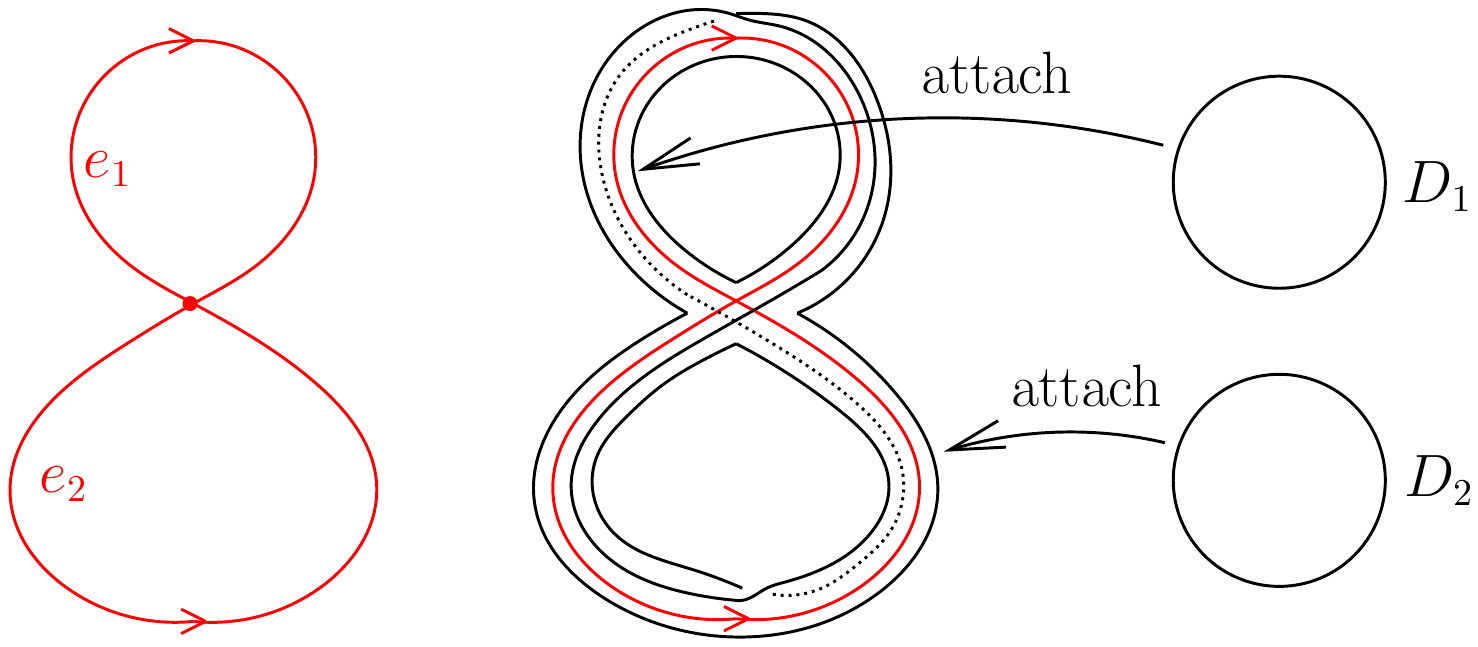}
\caption{The abalone $A$.}\label{fig11}
\end{center}
\end{figure}

We give an S-stable foliation on $A$ defined by a $1$-form $\beta$ with $d\beta>0$ explicitly.
Set 
$(E_1,E_2)=(-1,13)\in C(A)$. 
According to the construction in the proof of Theorem~\ref{thm1}, one can construct an S-stable foliation on $\Nbd(S(A);A)$ defined by a $1$-form $\beta_0$ with $d\beta_0>0$ as the left two figures in Figure~\ref{fig12}. 
Let $R_1$ and $R_2$ be the regions of 
$A$ containing $D_1$ and $D_2$, respectively. 
We describe four leaves on $R_2$ as shown on the right in Figure~\ref{fig12}.
We denote these leaves by $\ell_1,\ldots,\ell_4$.
For each $\ell_i$, we can find a $1$-form $\beta_{\ell_i}$ on $\Nbd(\ell_i;R_2)$ so that 
$\beta_{\ell_i}=\beta_0$ on $\Nbd(\ell_i;R_2)\cap\Nbd(\partial R_2;R_2)$
and $\ell_i$ is a leaf of the foliation defined by $\beta_{\ell_i}=0$.

\begin{figure}[htbp]
\begin{center}
\includegraphics[width=14.0cm, bb=129 554 588 710]{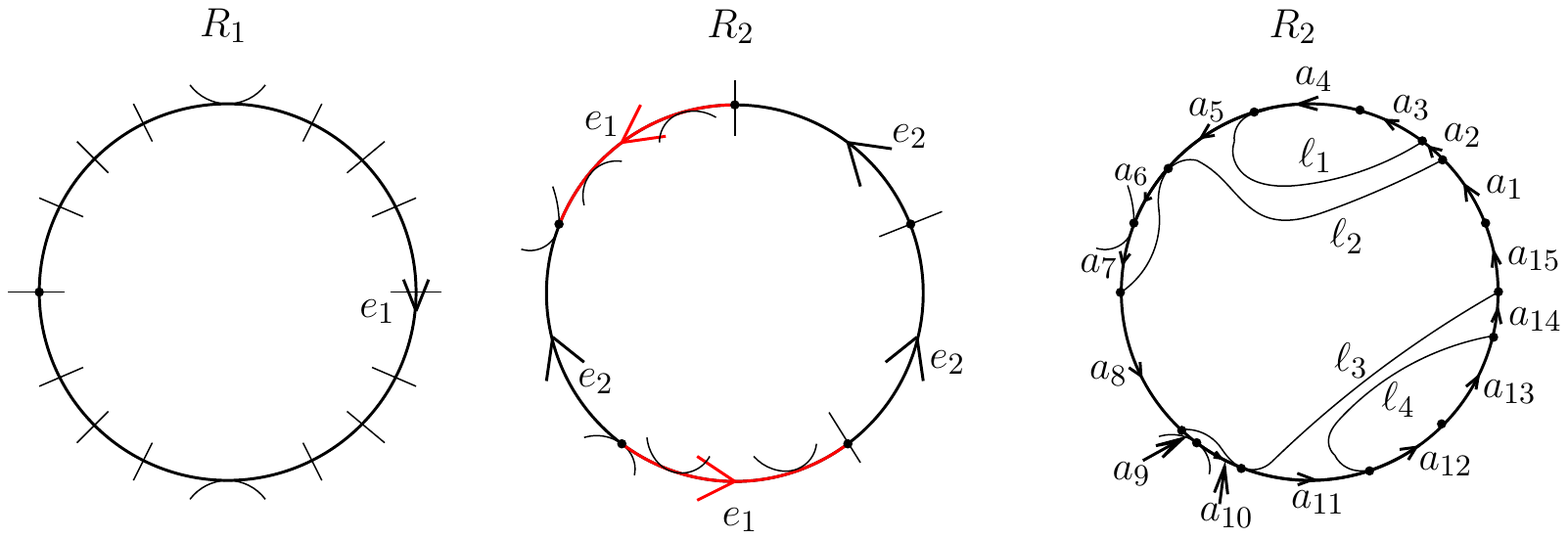}
\caption{An S-stable foliation on $\Nbd(S(A);A)$ (left and middle) and the leaves $\ell_1,\ldots,\ell_4$ on 
$R_2$ (right).}\label{fig12}
\end{center}
\end{figure}

We decompose $\partial R_2$ into $15$ arcs by the endpoints of the leaves $\ell_1,\ldots,\ell_4$, the tangent points of these leaves with $\partial R_2$ and the vertices on $\partial R_2$. 
We label these arcs as $a_1,\ldots,a_{15}$ as shown on the right in Figure~\ref{fig12}.
For the arcs $a_1,\ldots,a_{15}$, we assign real numbers $A_1,\ldots,A_{15}$ so that
\[
\begin{split}
& A_1+A_2+A_3=E_2, \hspace{9.5mm} A_4+A_5+A_6=E_1, \hspace{7mm} A_7+A_8+A_9=-E_2, \\
& A_{10}+A_{11}+A_{12}=E_1, \hspace{5mm} A_{13}+A_{14}+A_{15}=E_2
\end{split}
\]
and
\[
\begin{split}
&A_{15}+A_1+A_8>0, \hspace{5mm} 
A_2+A_5>0, \hspace{7.7mm} 
A_3>0, \hspace{7mm} 
A_4>0, \hspace{7mm} 
A_6+A_7>0, \\
&A_9+A_{10}>0, \hspace{14.2mm} 
A_{11}+A_{14}>0, \hspace{5mm} 
A_{12}>0, \hspace{5.5mm} 
A_{13}>0,
\end{split}
\]
and $(A_4, A_5, A_6)=(A_{10}, A_{11}, A_{12})$. 
The five equalities mean that the assignment of real numbers to $a_1,\ldots,a_{15}$ is a refinement of the assignment of those to $e_1$ and $e_2$.
The nine inequalities mean that, for each region bounded by $\partial R_2\cup\bigcup_{i=1}^4\ell_i$, the sum of the real numbers along its boundary is positive,
which is necessary for applying Lemma~\ref{lemma_TW}. 
Note that we set the real number assigned to $\ell_i$ to be $0$ since $\ell_i$ will be a leaf of the foliation
obtained by a $1$-form.
The last equality is needed since, for each $i=4,5,6$, the two edges $a_i$ and $a_{i+6}$ correspond to the same arc on the edge $e_1$ divided by the two simple tangency points. 
This system of equalities and inequalities has a solution, for example as
\[
   (A_1,\ldots,A_{15})=(6,6,1,2,-5,2,-1,-11,-1,2,-5,2,1,6,6). 
\]
Let $\beta'$ be a $1$-form defined on a small neighborhood $N$ of $S(P)\cup (\cup_{i=1}^4 \ell_i)$ in $P$ so that 
\begin{itemize}
\item$d\beta'>0$ on $N$,
\item the characteristic foliation on $N$ has leaves shown on the right in Figure~\ref{fig12},
\item $\int_{a_j}\beta'=A_j$ for $j=1,\ldots,15$.
\end{itemize}
Such a $1$-form can be easily constructed as in the proof of Theorem~\ref{thm1} (cf.~\cite[the proof of Lemma 6.3]{IIKN20}). We may extend $\beta'$ to the whole $P$ by applying Lemma~\ref{lemma_TW}.

\begin{remark}\label{rem62}
The arcs $\ell_1,\ldots,\ell_4$ are {\it Legendrian} if we consider a contact form $\alpha$ on $\Nbd(A;M)$ whose characteristic foliation coincides with the foliation defined by $\beta'=0$. 
When we describe a characteristic foliation of a flow-spine with a transverse Reeb flow, we need to choose the positions of these leaves so that the integration of $\beta'$ along the boundary of each region
separated by these leaves is positive. 
This gives a strong restriction for positions of leaves, which is different from the case of characteristic foliations on surfaces. 
\end{remark}

Consider a foliation on $A$ shown in Figure~\ref{fig13}, which 
coincides with the one defined by $\beta'$ on $N$. 
It has two elliptic singularities with positive divergence and no hyperbolic singularities (cf. Lemma~\ref{lemma_PH}). 
It is not difficult to find a $1$-form $\beta$ on $P$ with $d\beta>0$ that defines the foliation in the figure. 
For example, let $R$ be the region on $R_2$ bounded 
by the leaves $\ell_3$ and $\ell_4$ and the edges $a_{11}$ and $a_{14}$ 
and regard $R$ as a rectangle with coordinates $(x,y)$ so that the foliation is parallel to the $x$-axis.
On $\Nbd(\partial R;R)$, $\beta$ is given in the form $\varphi(x,y)dy$, where $\varphi(x,y)$ is a smooth function 
on $\Nbd(\partial R;R)$ with $\frac{\partial \varphi}{\partial x}(x,y)>0$. 
Remark that 
$\int_{a_{11}}\varphi(x,y)dy=A_{11}$, $\int_{a_{14}}\varphi(x,y)dy=A_{14}$ and $A_{11}+A_{14}>0$.
We observe the graph of the function $\varphi(x,y)$ on $\Nbd(\partial R;R)$,
and extend it to the whole $R$ so that $\frac{\partial \varphi}{\partial x}(x,y)>0$.
The $1$-form $\varphi(x,y)dy$ defines the foliation on $R$ and satisfies $d(\varphi(x,y)dy)>0$. 

The region $R'$ bounded by the leaf $\ell_4$ and the edges $a_{12}$ and $a_{13}$ has one elliptic singularity with positive divergence.
For any region $T$ bounded by an arc on $a_{12}\cup a_{13}$ and
two leaves connecting the elliptic singularity and the endpoints of the arc,
the $1$-form $\beta$ that we are going to make should satisfy $\int_T d\beta>0$. 
Therefore, the assignment to the edges $a_{12}$ and $a_{13}$ must satisfy the inequalities $A_{12}>0$ and $A_{13}>0$. Under this setting, the $1$-form on $R'$ with required property can be found by the same way as before.

Applying the same construction for other regions bounded by the E-cycles and the leaves $\ell_1,\ldots,\ell_4$, 
we can obtain a $1$-form $\beta$ on $A$ that defines the foliation in Figure~\ref{fig13} and satisfies $d\beta>0$.
The DS-diagram with this S-stable foliation is described in Figure~\ref{fig17}. 

\begin{figure}[htbp]
\begin{center}
\includegraphics[width=10.0cm, bb=129 496 544 710]{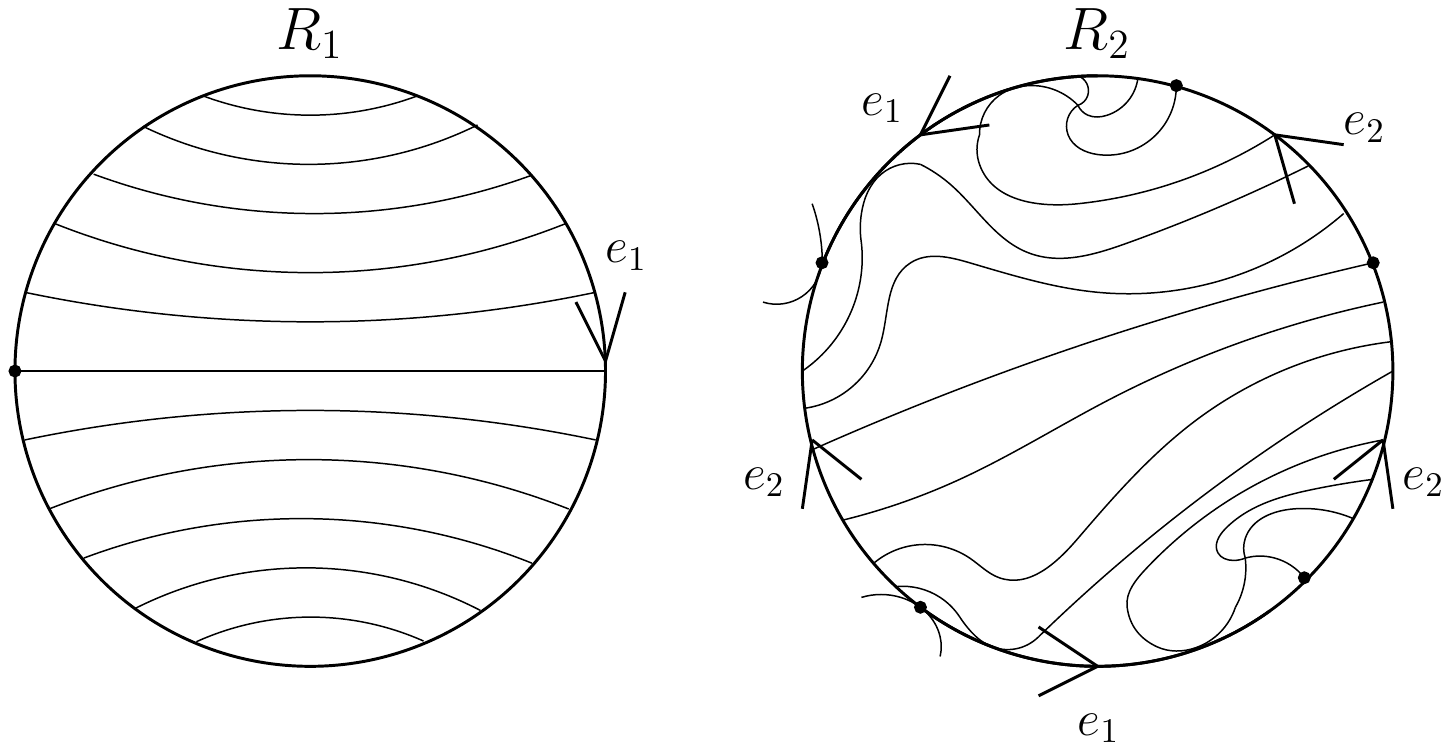}
\caption{An S-stable foliation on $A$ with $d\beta>0$ and with minimal number of simple tangency points.}\label{fig13}
\end{center}
\end{figure}

\begin{figure}[htbp]
\begin{center}
\includegraphics[width=6.0cm, bb=182 370 394 712]{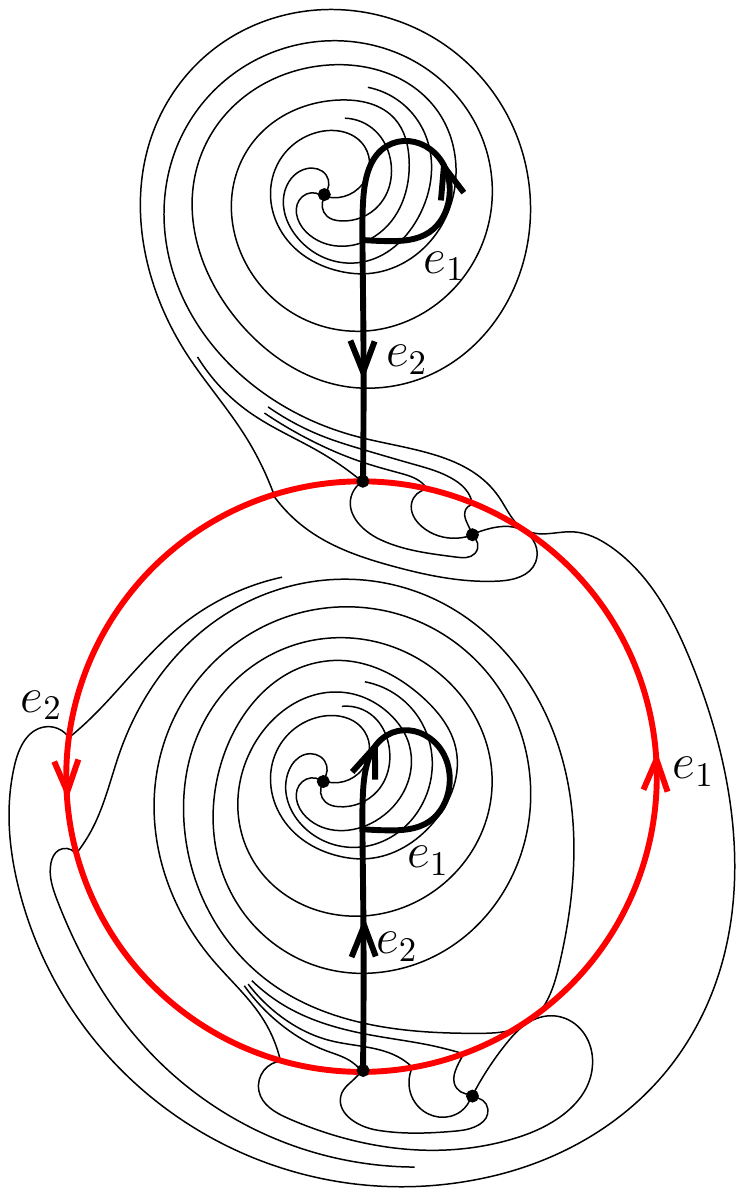}
\caption{The DS-diagram with the S-stable foliation in Figure~\ref{fig13}.}\label{fig17}
\end{center}
\end{figure}

The foliation in Figure~\ref{fig13} has two simple tangency points on $e_1$. 
Since $R_1$ is a preferred region,
due to Theorem~\ref{thm2}, we can conclude that there is no S-stable foliation defined by a $1$-form $\beta$ with $d\beta>0$ and with less simple tangency points.

A characteristic foliation of the flow-spine of the Poincar\'{e} homology $3$-sphere whose DS-diagram is 
the dodecahedron is observed by the first author in~\cite{Han18} by the same manner.

\subsection{An example without simple tangency point}\label{ex3}

Let $P$ be the branched simple polyhedron obtained from the branched polyhedron with boundary described in Figure~\ref{fig9} by attaching four disks along the four connected components of the boundary of $P$. It has three vertices, six edges and four disk regions. 
The orientations of the edges $e_1, \ldots, e_6$ in the figure are those 
induced from the adjacent regions by the rule in Figure~\ref{fig2}. 
The Euler characteristic of $P$ is $3-6+4=1$, and hence 
the Euler characteristic of the boundary of a thickening of $P$ is $2$. 
We may easily verify that the boundary of a thickening of $P$ is connected and hence it is $S^2$. 
This means that $P$ is a branched standard spine of a closed, oriented $3$-manifold.
We can check that the DS-diagram of this spine, forgetting the branching, is the diagram (1-10) in~\cite{Ish85} and conclude that the $3$-manifold is $S^3$. 

\begin{figure}[htbp]
\begin{center}
\includegraphics[width=7.0cm, bb=129 464 480 706]{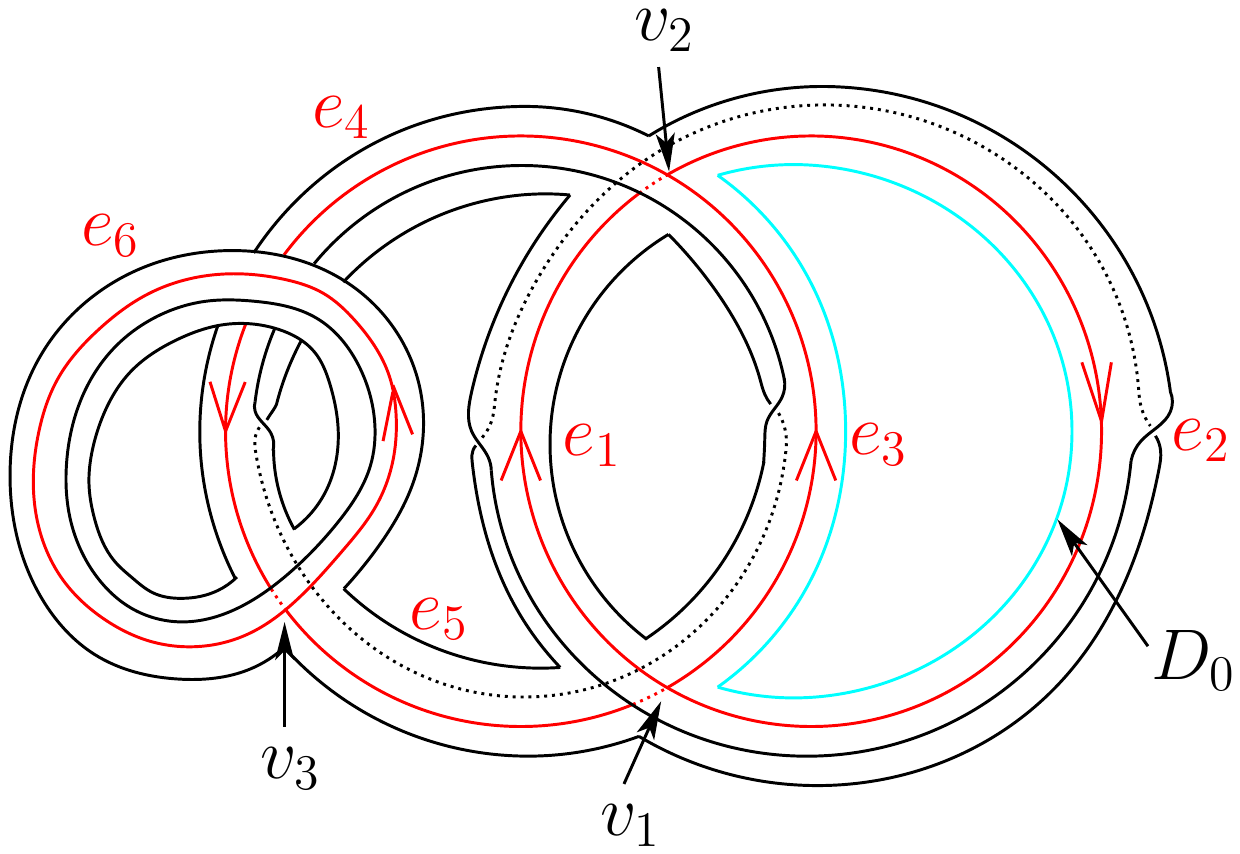}
\caption{A branched standard spine that admits an S-stable foliation defined by a $1$-form $\beta$ with $d\beta>0$ and without simple tangency points.}\label{fig9}
\end{center}
\end{figure}

The region containing the disk $D_0$ in the figure is a preferred region.
The defining inequalities of $C(P)$ are
$E_6>0$, $-E_2-E_3>0$, $E_1+E_2-E_6>0$ and $E_2+2E_3+E_4+E_5+E_6>0$.
This system of inequalities has a solution, for example as 
$(6,1,-2,-1,-1,6)\in\Real^6$. Hence $P$ is admissible. 
We choose the projections $\pr_{v_i}$ for each of vertices $v_1$, $v_2$ and $v_3$
so that their H-pieces become as shown in Figure~\ref{fig10}.
Under this setting, we can obtain an $S$-stable foliation defined by a $1$-form $\beta$ with $d\beta>0$ and without simple tangency points.

\begin{figure}[htbp]
\begin{center}
\includegraphics[width=9.0cm, bb=129 562 533 707]{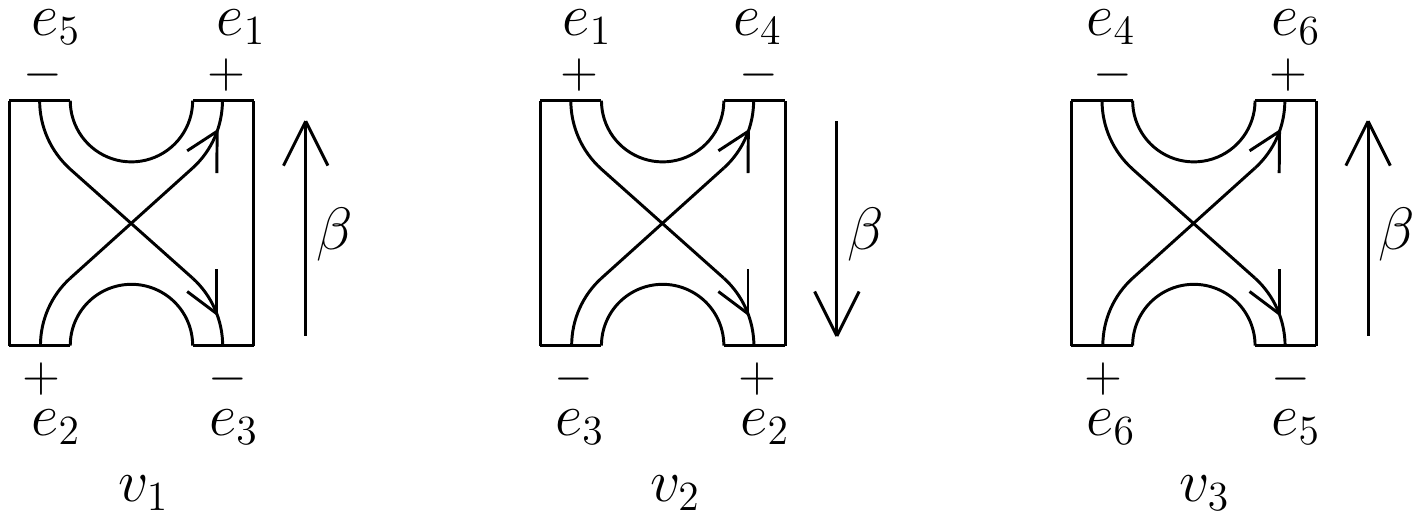}
\caption{The H-pieces of $v_1$, $v_2$ and $v_3$.}\label{fig10}
\end{center}
\end{figure}

We here give an S-stable foliation on $P$ defined by a $1$-form $\beta$ with $d\beta>0$ and without simple tangency points on $P$ explicitly.
Set
\[
   (E_1, E_2, E_3, E_4, E_5, E_6)=(6,1,-2,-1,-1,6)\in C(P).
\]
Let $R_1, R_2, R_3, R_4$ be the regions of $P$ bounded by the sequences of edges $\{e_6\}$,
 $\{-e_2, -e_3\}$, $\{e_1,e_2,-e_5,-e_6,e_5\}$ and $\{e_1,e_4,e_5,e_3,-e_1,e_3,e_4,e_6,-e_4,e_2\}$, respectively, see Figure~\ref{fig16}. Draw four leaves $\ell_1, \ell_2, \ell_3$ and $\ell_4$ as shown in the figure.
The edge $e_1$ on the boundary of $R_3$ divides into two oriented arcs by an endpoint of $\ell_1$, named $a_1$ and $a_2$ as in the figure. Similarly, we name $a_3, a_4, a_5, a_6, a_7$ and $a_8$ as in the figure. For these arcs $a_1,\ldots,a_8$, we assign real numbers $A_1,\ldots,A_8$ as
\[
   (A_1,A_2, A_3, A_4, A_5, A_6, A_7, A_8)=(1.6, 4.4, 0.8, 5.2, 1.5, 4.5, 3.2, 2.8)
\]
so that they satisfy $A_1+A_2=A_5+A_6=E_1$, $A_3+A_4=A_7+A_8=E_6$
and 
\[
\begin{split}
&E_5+A_1>0, \hspace{32mm}
A_2+E_2-A_4>0, \hspace{16mm}
-E_5-A_3>0,\\
&E_1+E_4+E_5+E_3-A_5>0, \hspace{4mm}
-A_6+A_8-E_4+E_2>0, \hspace{4mm}
E_3+E_4+A_7>0. 
\end{split}
\]
Thus, this assignment satisfies the required conditions for obtaining a $1$-form $\beta$ with $d\beta>0$ on $P$.
We may draw a foliation in the regions $R_1,\ldots, R_4$ described in Figure~\ref{fig16} so that
$R_1$ has only one elliptic singularity in the middle 
and the other regions have no singularity, and there exists a $1$-form $\beta$ with $d\beta>0$ on $P$ which gives this foliation as in the case of abalone.
There are two ``tangency points'' on the boundary of $R_3$ described in Figure~\ref{fig16}, one is at the intersection of the edges $e_5$ and $e_1$ and the other is of the edges $e_5$ and $e_6$,
and the boundary of $R_4$ also has two ``tangency points''.
The ``tangency point'' at the intersection of $e_5$ and $e_1$ appears in a neighborhood of the vertex $v_1$ shown on the left in Figure~\ref{fig10}, where the leaves of the foliation given by $\beta=0$ is horizontal. The other three ``tangency points'' appear by the same reason. 
Therefore, these leaves are transverse to the singular set of $P$
and the foliation given in Figure~\ref{fig16} has no simple tangency points.
The DS-diagram with this S-stable foliation is described in Figure~\ref{fig17}.

\begin{figure}[htbp]
\begin{center}
\includegraphics[width=10.0cm, bb=137 431 436 712]{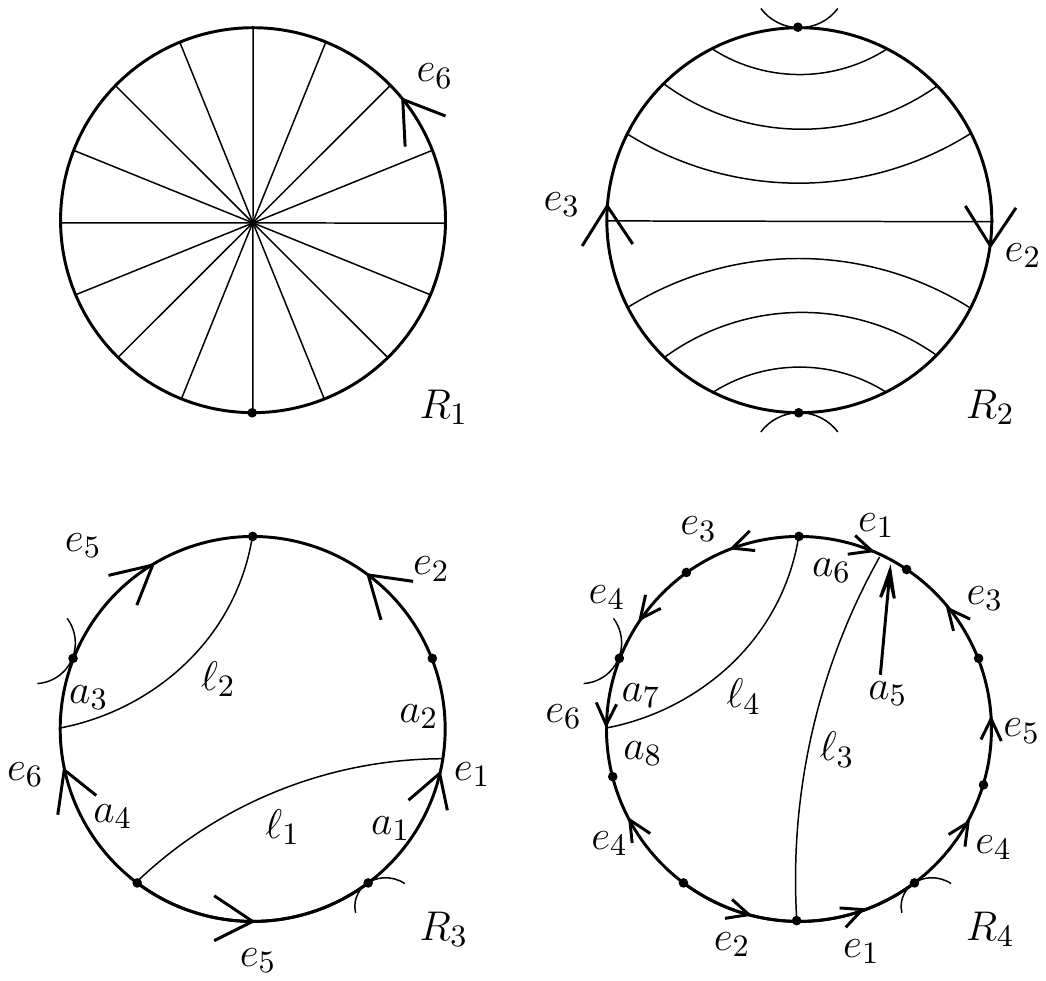}
\caption{An S-stable foliation on $P$ with $d\beta>0$ and without simple tangency points.}\label{fig16}
\end{center}
\end{figure}

\begin{figure}[htbp]
\begin{center}
\includegraphics[width=9.0cm, bb=137 339 457 712]{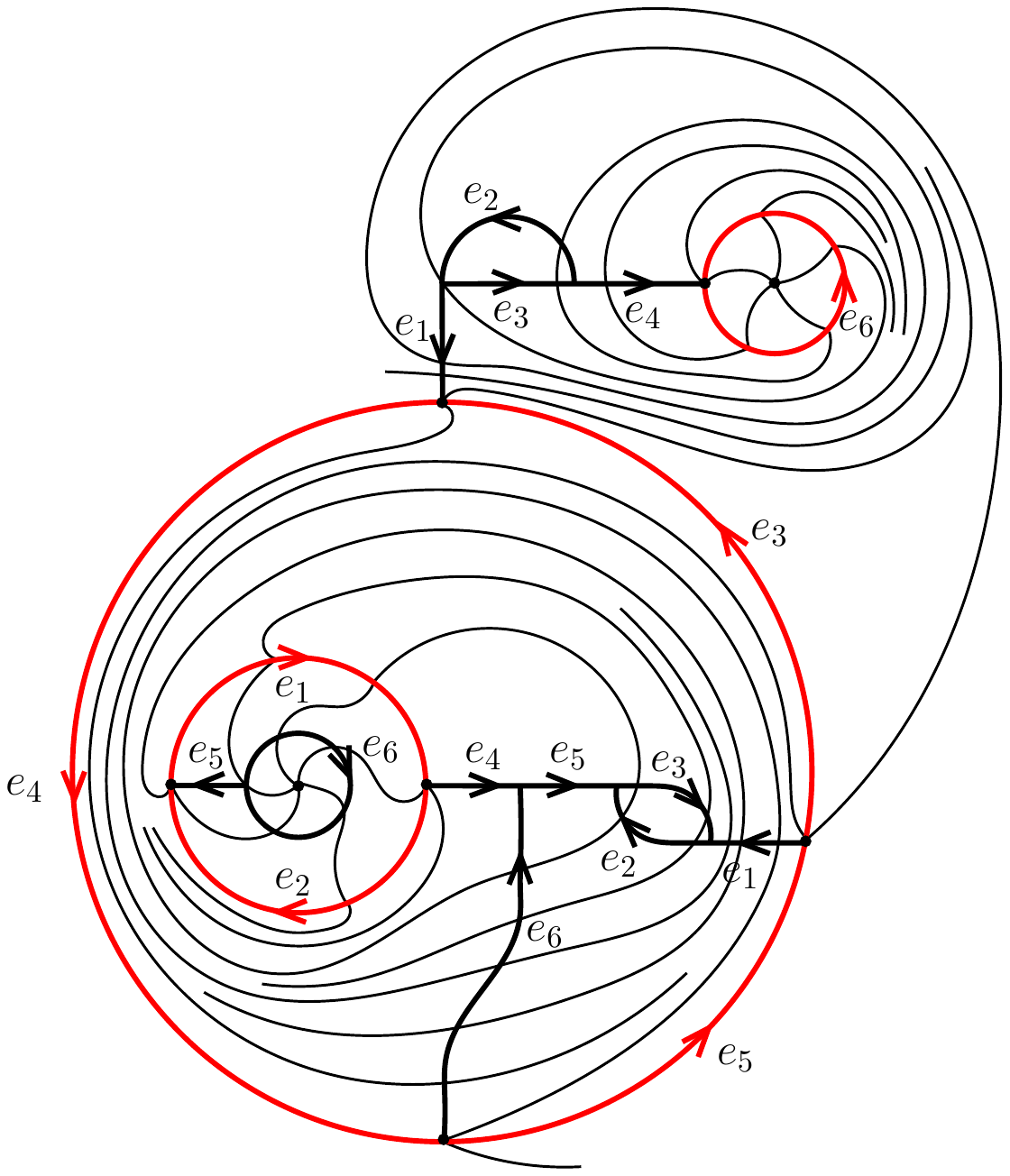}
\caption{The DS-diagram with the S-stable foliation in Figure~\ref{fig16}.}\label{fig18}
\end{center}
\end{figure}

\end{document}